\documentclass[12pt]{amsart}
\usepackage{amsmath,amssymb,fullpage}

\newcommand{\bb}[1]{\mathbb{#1}}
\newcommand{\cl}[1]{\mathcal{#1}}

\title[Factorization with respect to tensor products of nest algebras]{Reverse Cholesky factorization and tensor products of nest algebras}

\author[Paulsen]{Vern I. Paulsen}
\address{Institute for Quantum Computing and Department of Pure Mathematics\\
University of Waterloo\\
200 University Avenue West\\
Waterloo, ON, Canada  N2L 3G1}
\email{vpaulsen@uwaterloo.ca}
\author[Woerdeman]{Hugo J.~Woerdeman}
\address{Department of Mathematics \\
Drexel University\\
3141 Chestnut St.\\
Philadelphia, PA, 19104}
\email{hugo@math.drexel.edu}

\thanks{VP is partially supported by an NSERC grant. HW was partially supported by Simons Foundation grant
355645, and the Institute for Quantum Computing at the University of Waterloo. }

\theoremstyle{plain}

\newtheorem{thm}{Theorem}[section]
\newtheorem{cor}[thm]{Corollary}
\newtheorem{lem}[thm]{Lemma}

\newtheorem{prop}[thm]{Proposition}

\numberwithin{equation}{section}

\newcommand{\mc}{\mathcal}

\newcommand{\beq}{\begin{equation}}
\newcommand{\eeq}{\end{equation}}

\theoremstyle{remark}

\newtheorem{ex}[thm]{Example}

\numberwithin{equation}{section}

\newcommand{\bbm}{\begin{bmatrix}}
\newcommand{\ebm}{\end{bmatrix}}

\begin{document}

\begin{abstract} We prove that every positive semidefinite matrix over the natural numbers that is eventually 0 in each row and column can be factored as the product of an upper triangular matrix times a lower triangular matrix.
We also extend some known results about factorization with respect to tensor products of nest algebras. Our proofs use the theory of reproducing kernel Hilbert spaces.
\end{abstract}

\maketitle

\section{Introduction}\label{sec:Intro}
It is well-known that if $P=(p_{i,j})_{i,j \in \bb N}$ is the matrix representation of a bounded, positive semidefinite operator on the Hilbert space $\ell^2(\bb N)$, then $P$ can be factored as $P=LL^*$, where $L$ is lower triangular. This is often called the Cholesky factorization and can be obtained by applying the Cholesky algorithm, see \cite{PR}. Somewhat surprisingly, not every such $P$ can be factored as $P= UU^*$ with $U$ upper triangluar. 

Nonetheless, this latter type of upper-lower factorization is often important.  For example suppose that $P=T_p$ is the Toeplitz matrix whose symbol is a positive function $p$ on the unit circle. A classic result in function theory, says that there is a factorization $p = |f|^2$ with $f$ analytic on the disc if and only if $log(p)$ is integrable. This yields an upper-lower factorization, $T_p= T_f^*T_f$.
Conversely, it is known that  $T_p$ has an upper-lower factorization if and only if the symbol $p$ admits such an analytic factorization.  Thus, upper-lower factorizations for Toeplitz matrices is intimately related to the classical theory of analytic factorization. These results are discussed fully in \cite{AFMP}.

These considerations lead the authors of \cite{AMP} and  \cite{AFMP} to consider upper-lower factorizations. They proved that to each such positive matrix $P$, one could affiliate a reproducing kernel Hilbert space, and then they gave necessary and sufficient conditions in terms of properties of that Hilbert space for $P$ to have an upper-lower factorization.

When $R=(r_{i,j})_{1 \le i,j \le n}$ is a finite positive semidefinite matrix, then by implementing the Cholesky algorithm starting with the last entry, one obtains a factorization $R= UU^*$ with $U$ upper triangluar. For this reason a factorization of $P=UU^*$ with $U$ upper triangular is often referred to as a {\it reverse Cholesky factorization}. In fact, work on this topic prior to \cite{AFMP} often used the method of truncating $P$ at some point $n$, call it $P_n$, so that there was a last entry, factoring $P_n = U_nU_n^*$, then letting $n \to +\infty$ and imposing conditions to guarantee that these $U_n$'s possessed some type of limit.

In this paper we refine the reproducing kernel Hilbert space results somewhat and then apply our refinement to two situations. The first result that we obtain is that any matrix $P$ as above that is eventually 0 in each row and column, has an upper-lower factorization.  This extends considerably the well-known result that any positive matrix with finite bandwidth has an upper-lower factorization.

The second result uses the multi-variable analogues of these results, which were also developed in \cite{AMP} and \cite{AFMP}, and applies them to extend some results of \cite{AK1} and \cite{AK2} about factorization in tensor products of nest algebras.

Let $\cl H$ be a Hilbert space, $\cl B(\cl H)$ be the Banach algebra of bounded operators on $\cl H$, and let $\cl N$ be a {\it nest} of orthogonal projections in ${\cl B} ({\cl H})$ in the sense of \cite{Dav2}. Thus, ${\cl N}$ is a strongly closed, linearly ordered collection of projections on $\cl H$, containing $0$ and the identity $I$. The {\em nest algebra} corresponding to $\cl N$ is defined as 
$$ {\rm Alg} \ {\cl N} := \{ A \in {\cl B}({\cl H} ) : P^\perp A P =0 \ P \ {\rm for \ all} \ P\in {\cl N } \} , $$
where $P^\perp = I-P$. 
The question of finding a factorization $A=BB^*$ with $B\in  {\rm Alg} \ {\cl N}$ for a positive definite $A$, goes back to Arveson \cite{Arv}. Accounts of the various results since may be found in \cite{Dav}, \cite{AK1}, \cite{AK2}.  In this paper we study the factorization problem where $B$ is required to lie in a tensor product of nest algebras.

For Hilbert spaces ${\cl H}_1, \ldots , {\cl H}_d$ we let ${\cl H}_1\otimes \cdots \otimes {\cl H}_d$ denote their Hilbertian tensor product. If ${\cl A}_i \subseteq {\cl H}_i$, $i=1\ldots , d $, we let ${\cl A}_1\otimes \cdots \otimes {\cl A}_d$ denote the weakly closed subalgebra of ${\cl B} ( {\cl H}_1\otimes \cdots \otimes {\cl H}_d )$ generated by elementary tensors ${ A}_1\otimes \cdots \otimes {A}_d$, where $A_i \in {\cl A}_i$, $i=1,\ldots , d$. In \cite{GHL} tensor products of nest algebras were studied, where among other things it was shown that for nests ${\cl N}_i \subset {\cl B}({\cl H}_i )$, $i=1,\ldots , d $, 
$$ {\rm Alg} \ ({{\cl N}_1\otimes \cdots \otimes {\cl N}_d}) = {\rm Alg} \ {{\cl N}_1\otimes \cdots \otimes {\rm Alg}\ {\cl N}_d} . $$

We consider the question of when a positive semidefinite operator $Q \in {\cl B} ( {\cl H}_1\otimes \cdots \otimes {\cl H}_d )$ allows a factorization $Q=BB^*$ with $B \in  {\rm Alg} \ ({{\cl N}_1\otimes \cdots \otimes {\cl N}_d})$. 
There are certainly nests for which this fails in general. For example, if ${\cl N} = \{ 0 , \begin{pmatrix} 1 & 0 \cr 0 & 0 \end{pmatrix} , I \} \subset {\cl B} ({\bb C}^2)$ and $U=(U_{ij})_{i,j=1}^4$ is an upper triangular invertible matrix with $U_{23}\neq 0$, then $UU^*$ cannot be written as $BB^*$, with $B \in  {\rm Alg} \ ({{\cl N}\otimes {\cl N}})$. Indeed, the upper-lower reverse Cholesky factorization is unique up to a right multiplication with a diagonal unitary $D$, and $UD \not \in  {\rm Alg} \ ({{\cl N}\otimes {\cl N}})$. We will show, however, that if ${\cl N}$ is of order type ${\mathbb N}$, then any invertible, positive definite $Q$ factors as $BB^*$ with $B \in  {\rm Alg} \ ({\cl N}^{\otimes d} )$. The key difference is that when the nest is infinite, then we can use a Hilbert hotel type of argument.

\section{Factorization and Multi-variable Reproducing Kernel HIlbert Spaces}

In this section we recall the results of \cite{AFMP} connecting upper-lower factorization of operators with properties of an affiliated reproducing kernel Hilbert space.
 For a general reference on these spaces see either \cite{Aro} or \cite{PR}. We use the set up from \cite{AMP} and \cite{AFMP}. Let ${\mc C}$ be a Hilbert space and $G\subseteq {\mathbb C}^d$ be an open neighborhood of $0$. Let $K: G\times G \to {\mc B} ({\mc C} )$ be analytic in the first variable and co-analytic in the scond variable, such that $K$ is {\it positive}, i.e.,
$$ \sum_{i,j=0}^n \langle K(z_i,z_j) x_j , x_i \rangle_{\mc C} \ge 0 \ {\rm for \ all}\ n\in{\mathbb N}, z_1,\ldots , z_n \in G \ {\rm and} \ x_1, \ldots , x_n \in {\mc C} . $$ 
With $K$ we associate a reproducing kernel Hilbert space consisting of analytic functions on $G$, which is the completion of $$ \{ f(z) = \sum_{j=1}^n K(z,w_j) x_j : n \in {\mathbb N}, w_1, \ldots , w_n \in G, x_1, \ldots , x_n \in {\mc C} \}  $$ with inner product defined via
$$ \langle K(z,w_1)x_1 , K(z,w_2) x_2 \rangle = \langle K(w_1, w_2)x_1 , x_2 \rangle_{\mc C} . $$
Let ${\mathbb N}_0=\{ 0, 1, \ldots \}$ and $d \in {\mathbb N}$. Then ${\mathbb N}_0^d$ is partially ordered by setting $I=(i_1, \ldots , i_d ) \ge (j_1, \ldots , j_d ) = J$ if and only if $i_k \le j_k$, $k=1,\ldots , d $. If $z=(z_1, \ldots , z_d) \in {\mathbb C}^d$ we set $\bar{z} = (\bar{z}_1, \ldots , \bar{z}_d )$ and $z^I = z_1^{i_1} \cdots z_d^{i_d}$. If $Q=(Q_{I,J})_{I,J \ge 0}$, $Q_{I,J} \in {\mc B} ({\mc C})$, is positive semidefinite on finite sections and $K(z,w):= \sum_{I,J \ge 0} z^I\bar{w}^J Q_{I,J}$ is convergent on some polydisk, then by results of \cite{AMP}, $K(z,w)$ is positive on that polydisk. 

Let $Q \in {\mc B} (\ell^2({\mc C})^{\otimes d} )$, where we identify $\ell^2 = \ell^2(\bb N_0)$ with orthonormal basis $\{ e_i: i \in \bb N_0\}$. As explained in \cite{AFMP}, we may represent $Q$ in a standard way as $Q=(Q_{I,J})_{I,J \ge 0}$ where $Q_{I,J} \in {\mc B} ({\mc C})$. Briefly, an orthonormal basis of $(\ell^2)^{\otimes d}$ is given by $\{ e_I: I \ge 0 \}$ where $e_I:= e_{i_1} \otimes \cdots \otimes e_{i_d}$.
Hence, $\ell^2({\mc C})^{\otimes d}$ can be identified with the direct sum of copies of $\cl C$ indexed by $I \in \bb \bb N_0^d$ and $Q_{I,J}$ is the restriction of $Q$ to $\cl C_J$ followed by projection onto $\cl C_I$. 

If $Q$ is also positive semidefinite, then \cite{AMP} shows that $K(z,w):= \sum_{I,J \ge 0} z^I\bar{w}^J Q_{I,J}$ converges for $z,w \in {\mathbb D}^d$, and yields an associated reproducing kernel Hilbert space of analytic ${\mc C}$-valued functions on ${\mathbb D}^d$, denoted by ${\mc H} (Q)$. 

A function $f: {\mathbb C}^d \to {\mc C}$ is called a {\it polynomial} if there exists a finite collection of vectors $v_I \in {\mc C}$, $I \in {\mathbb N}_0^d$, so that $f(z) = \sum v_I z^I$. The {\it degree} of $f$ is $\max \{ |I| : v_I \neq 0 \}$, where $|I| = |(i_1, \ldots , i_d ) | := i_1+ \cdots + i_d$.
For $J=(j_1, \ldots , j_d)$ we denote, as usual, $J!=j_1!\cdots j_d!$ and $$ \frac{\partial^J}{\partial z^J} = \frac{\partial^{|J|}}{\partial z_1^{j_1} \partial z_2^{j_2} \cdots \partial z_d^{j_d}}. $$

\begin{prop}\label{pp} Let $Q=(Q_{I,J})_{I,J \ge 0} \in {\mc B} (\ell^2({\mc C})^{\otimes d} )$ be positive semidefinite. For $J \in {\mathbb N}_0^d$ introduce the linear operator $L_J : {\mc H} (Q) \to {\mc C}$ defined by $$L_J (f) := \frac{\partial^J}{\partial z^J} f |_{z=0} = f^{(J)} (0)$$ and for $v \in \cl C$ set 
set $\phi_{J,v} (z) =\sum_{I\ge 0} Q_{I,J}v z^I$. 
Then
\begin{enumerate}
\item the operators $L_J$ are bounded,
\item $\phi_{J,v} \in {\mc H} (Q)$, 
\item \label{fJ} $\langle f , (J!)\phi_{J,v} \rangle_{{\mc H} (Q)} = \langle L_J(f) , v \rangle_{\mc C}$ , 
\item
 $\frac{1}{J!} \|L_J\|=  \sup_{\| v \| =1 } \|\phi_{J,v} \| \le  \min \{c \ge 0 :  (Q_{I,J} Q_{K,J}^*)_{I,K\ge 0} \le c^2 Q \}  \le  \|Q\|^{1/2}, $ 
 \item the span of the $\phi_{J,v}$'s is dense in ${\cl H} (Q)$.
\end{enumerate}
\end{prop}
\begin{proof} 
First we observe that $\phi_{J,v} \in {\mc H} (A)$  if and only if $\phi_{J,v} (z) \phi_{J,v}(w)^* \le c^2 K(z,w)$
%
%
in the order on kernel functions. This translates into
$$ (Q_{I,J} \Pi_v Q_{K,J}^*)_{I,K\ge 0} \le c^2 Q , $$ where $\Pi_v$ is the orthogonal projection of ${\mc C}$ onto ${\rm span} \{ v \}$. 
%
%
%
%
 Note that since $Q=Q^*$,  $Q_{I,J} = Q_{J,I}^*$. Now, $$ (Q_{I,J} \Pi_v Q_{K,J}^*)_{I,K\ge 0} \le  (Q_{I,J} Q_{K,J}^*)_{I,K\ge 0} \le \sum_{J\ge 0} (Q_{I,J} Q_{J,K})_{I,K\ge 0} = Q^2 \le c^2 Q$$ with $c^2 = \|Q\|$.  Hence,  $\phi_{J,v} (z) \phi_{J,v}(w)^* \le c^2 K(z,w)$ and $\phi_{J,v} \in {\cl H} (Q)$ follows. Moreover, $\| \phi_{J,v} \| \le \| Q^{\frac12} \| $.

From the equality 
$$ \langle f , K(\cdot , w)v \rangle_{{\mc H} (Q)} = \langle f(w) , v \rangle_{\mc C} , $$
we obtain that
$$ \langle L_J(f) , v \rangle_{\mc C} = \frac{\partial^J}{\partial w^J} \langle f(w) , v \rangle_{\mc C} |_{w=0} = \frac{\partial^J}{\partial w^J} \langle f , K(\cdot , w)v \rangle_{{\mc H} (Q)} |_{w=0} =
\langle f , J! \phi_{J,v}\rangle_{{\mc H} (Q)}. $$
Now from (\ref{fJ}) the equality $\frac{1}{J!} \|L_J\| =  \sup_{\| v \| =1 } \|\phi_{J,v} \|$ follows immediately. 

 
Finally, suppose that $f \in {\cl H} (Q)$ and $f \perp \phi_{J,v}$ for all $J \ge 0$ and $v\in {\cl C}$. Then $L_J(f) = 0$ for all $J\ge 0$, and thus all the coefficients of the Taylor series for $f$ are 0. But since $f$ is analytic on $\bb D^d$, this implies that $f$ is 0.
 \end{proof}

\section{Cholesky factorization with respect to tensor products of nest algebras}\label{CF}

Let ${\mc C}$ be a Hilbert space. On $\ell^2({\mc C}) = \{ (\eta_j )_{j=0}^\infty : \eta_j \in {\mc C}, \sum_{j=0}^\infty \| \eta_j \|^2 < \infty \}$, let the canonical projections $P_i \in {\mc B} (\ell^2({\mc C}) )$, $i=0,1, \ldots $, be defined by
$$ P_i [ (\eta_j )_{j=0}^\infty ] = \begin{pmatrix} \eta_0 \cr \vdots \cr \eta_i \cr 0 \cr \vdots \end{pmatrix} . $$ We let ${\cl N} = \{ 0 < P_0 < P_1 < P_2 < \cdots < I \}$ be the standard nest on 
$\ell^2({\mc C})$, and denote $\ell^2({\mc C})^{\otimes d} := \ell^2({\mc C}) \otimes \cdots \otimes \ell^2({\mc C})$ ($d$ copies). 

The following result generalizes the Cholesky factorization result with respect to a nest algebra (see \cite[Theorem 13]{AK2}) to the tensor product of the above nest.

For an operator ${\mc B} ({\mc C})$ we define the {\it range space} ${\mc R} (B)$ to be the Hilbert space one obtains by equipping the range of $B$ (denoted by ${\rm Ran} \ B$) with the norm $\| y \|_{{\mc R} (B)} := \| x \|_{\mc C}$, where $x$ is the unique vector in $({\rm Ker} \ B)^\perp$ so that $Bx=y$.

\begin{thm}\label{tensornest} Let $A \in {\mc B} (\ell^2({\mc C})^{\otimes d} )$. The following are equivalent.
\begin{itemize} \item[(i)] $AA^*=BB^*$ for some $B \in {\rm Alg} ({\mc N}^{\otimes d} )$.
\item[(ii)] ${\rm Ran}\ A = {\rm Ran }\ C$ for some $C \in {\rm Alg} ({\mc N}^{\otimes d} )$.
\end{itemize}
\end{thm}

\begin{proof} (i)$\implies$(ii). By Douglas' lemma \cite{Douglas} we have that $AA^*= BB^*$ implies that ${\rm Ran}\ A = {\rm Ran }\ B$.

(ii)$\implies$(i). If $A=0$ there is nothing to prove, so let us assume that $A\neq 0$. Suppose $C \in {\rm Alg} ({\mc N}^{\otimes d} )$ so that ${\rm Ran}\ A = {\rm Ran }\ C$. As $A\neq 0$, clearly $C \neq 0$ as well. By Douglas' lemma \cite{Douglas}, we have that there exist $\lambda , \mu \ge 0$ so that $AA^* \le \lambda CC^*$ and $CC^* \le \mu AA^*$. We observe that since $A\neq 0 \neq C$, it follows that $\lambda , \mu >0$. Now it follows that the range spaces ${\mc R} (A)$ and ${\mc R} (C)$ are equivalent (i.e., they contain the same elements and their norms are equivalent). By \cite[Corollary 3.3]{AMP} it follows that the reproducing kernel Hilbert spaces ${\mc H} (AA^*)$ and ${\mc H} (CC^*)$ are equivalent. Next, \cite[Theorem 3.1]{AFMP} yields that the polynomials in ${\mc H} (CC^*)$ are dense. But then the same holds for ${\mc H} (AA^*)$. Again applying \cite[Theorem 3.1]{AFMP}, now in the other direction, gives that $AA^* = BB^*$ for some $B \in {\rm Alg} ({\mc N}^{\otimes d} )$. \hfill $\square$
\end{proof}

\begin{cor} Let $Q \in {\mc B} (\ell^2({\mc C})^{\otimes d} )$ be invertible and positive definite. Then $Q=BB^*$ for some $B \in {\rm Alg} ({\mc N}^{\otimes d} )$.
\end{cor}

\begin{proof} Apply Theorem \ref{tensornest} with the choice $A=Q^{\frac12}$ and $C=I$.
\end{proof}

Using the further analysis of reproducing kernel Hilbert spaces stated in Proposition \ref{pp} we obtain the following additional result.

\begin{thm}\label{wrb} Let $Q=(Q_{I,J})_{I,J \ge 0} \in {\mc B} (\ell^2({\mc C})^{\otimes d} )$ be positive semidefinite. Suppose that for every $J \in {\mathbb N}_0^d$ we have that $Q_{I,J} \neq 0$ for only finitely many $I$. Then $Q=BB^*$ for some $B \in {\rm Alg} ({\mc N}^{\otimes d} )$.
\end{thm}

%
%
%
%
\begin{proof}
  Note that the fact that for every $J \in {\mathbb N}_0^d$ we have that $Q_{I,J} \neq 0$ for only finitely many $I$, implies that each $\phi_{J,v}$ is a polynomial. Since, by Proposition \ref{pp}, the span of the $\phi_{J,v}$'s is dense, the polynomials in ${\cl H} (Q)$ are dense.  The result now follows from \cite[Theorem 3.1]{AFMP}.
\end{proof}
  


\bigskip

{\it Acknowledgment.}
The research was conducted while the second author
was visiting the Institute for Quantum Computing at the University of
Waterloo. He gratefully acknowledges the hospitality of many at the
University of Waterloo.

\end{document}